\documentclass[12pt]{article}

\usepackage{enumerate}
\usepackage{amsmath}
\usepackage{amsfonts}
\usepackage{amssymb}
\usepackage{amsthm}
\usepackage[usenames,dvipsnames]{color}
\usepackage{dsfont}
\usepackage[all]{xy}

\usepackage[table]{xcolor}

\newcommand{\C}{\mathbb C}

\renewcommand{\P}{\mathbb P}

\newcommand{\Q}{\mathbb Q}
\newcommand{\Z}{\mathbb Z}
\newcommand{\N}{\mathbb N}
\newcommand{\D}{\mathbb D}
\newcommand{\E}{\mathbb E}
\newcommand{\cO}{\mathcal O}
\newcommand{\cB}{\mathcal B}

\newcommand{\cF}{\mathcal F}

\newcommand{\Ygen}{Y_{\mathrm{gen}}}

\renewcommand{\epsilon}{\varepsilon}

  \DeclareMathOperator{\Pf}{Pf}
 \DeclareMathOperator{\Proj}{Proj}

\DeclareMathOperator{\Sing}{Sing}

\newtheorem{thm}{Theorem}
\newtheorem{cor}[thm]{Corollary}
\newtheorem{lem}[thm]{Lemma}
\newtheorem{prop}[thm]{Proposition}

\theoremstyle{definition}

\theoremstyle{remark}

\theoremstyle{remark}
\newtheorem{eg}[thm]{Example}

\numberwithin{equation}{section} \numberwithin{thm}{section}

\newcommand{\QED}{\ifhmode\unskip\nobreak\fi\quad\ensuremath{\mathrm{QED}}}

\title{Polarized Calabi--Yau 3-folds in codimension 4}
\author{Gavin Brown and Konstantinos Georgiadis}
\date{}

\begin{document}
\maketitle

\begin{abstract}
We construct Calabi--Yau 3-folds as orbifolds embedded
in weighted projective space in codimension~4.
For each Hilbert series that is realised, there are at least two
different components of Calabi--Yau 3-folds.
\end{abstract}

\setcounter{tocdepth}{2}

\section{Introduction}

Throughout this paper, a {\em Calabi--Yau 3-fold} is a normal projective variety $X$
of dimension~3, with $K_X=0$ and $h^1(X,\cO_X)=0$ and with at worst canonical singularities.
In practice, we construct varieties whose singularities are cyclic quotients,
so $X$ is an orbifold, and moreover the singularities we consider admit crepant resolutions,
and every Calabi--Yau 3-fold we discuss has a crepant resolution by a Calabi--Yau manifold.

The classification of Calabi--Yau 3-folds is a central problem in birational geometry,
and one that seems very far from being solved.
We treat Calabi--Yau 3-folds as polarised varieties, $X,A$, including a choice of
ample $\Q$-Cartier divisor $A$. With this choice, we may regard $X$ as embedded
in weighted projective space by choosing homogeneous generators for the
graded ring
\[
R(X,A) = \bigoplus_{m\in\N} H^0(X,mA)
\]
and writing $X = \Proj R(X,A)$. The degrees $a_1,\dots,a_N$ of a set of minimal
generators of $R(X,A)$ are well defined, and so we write $X\subset\P^{N-1}(a_1,\dots,a_N)$.
The case $N=5$ of weighted hypersurfaces is classified by Kreuzer--Skarke \cite{KS}
into 184,026 distinct deformation families, of which 7,555 have orbifolds
as their general member.
There are lists of families in codimension~$\le3$, that is $N\le7$, in~\cite[Table~1]{BKZ}
and \cite{grdb},
although these are not complete classifications.

Our interest here is in Calabi--Yau 3-folds embedded in codimension~4,
the case $N=8$. We first prove the following as a model result.
\begin{thm}
\label{th!main1}
There are at least two deformation families
of quasismooth, projectively normal Calabi--Yau 3-folds
$X \subset \P^7(1,1,1,1,1,1,3,3)$,
nonsingular away from two quotient singularities, $2 \times \frac{1}{3}(1,1,1)$,
and with Hilbert series
\begin{displaymath}
P_{X}(t) = 
\frac{1 - 2t^3 - 6t^4 - t^6 + 6t^5 + 4t^6 + 6t^7 - t^6 - 6t^8 - 2t^9 + t^{12}}{(1-t)^6(1-t^3)^2}.
\end{displaymath}
\end{thm}
\noindent
The Hilbert series, expressed in this way, faithfully reports that the general such $X$ is
defined by 9 equations (two of degree~3, six of degree~4 and one of degree~6)
related by 16 syzygies (six of degree~5, etc.); that is, $R(X,A)$
has a $9\times 16$ first syzygy matrix, as is common for Gorenstein codimension~4.
The two families of this theorem are distinguished by their Euler characteristic, and so
in particular members of different families cannot be birational
(Corollary~\ref{cor!batyrev} below).
The central striking result that there are multiple deformation families in codimension~4
echoes both the same phenomenon for Fano 3-folds \cite{TJ} and Gross'
obstructedness result for Calabi--Yaus~\cite{gross}.

This is one of a large series of theorems, each one
constructing general members of two or more deformation families
of orbifold Calabi--Yau 3-folds embedding in codimension~4.
Theorem~\ref{main theorem} in \S\ref{s!web} below is a typical example.
The crucial point is the (quasi)smoothability of varieties in a Pfaffian format,
and \S\ref{s!comments} recalls a salutary example where this fails.
The rest of the introduction gives a geometric explanation of Theorem~\ref{th!main1}
and the techniques we use; then \S\S\ref{sec!pf1}--\ref{s!codim4}
give proofs of the results with additional comments in~\S\ref{s!comments}.
To reach codimension~4, we use double Gorenstein unprojection, 
with the novelty that unprojection divisors cannot
be chosen to be coordinate planes; see~\S\ref{s!codim4}.

\subsection{The web of Calabi--Yaus}
\label{s!web}

To find other codimension~4 embeddings,
in \S\ref{s!hilb} we begin an audit of rational functions which have
the same coarse numerical properties as Hilbert series of polarised
Calabi--Yau 3-folds; in \S\ref{s!codim3}--\ref{s!codim4}
we construct Calabi--Yau 3-folds that realise each case.
Already the fairly crude Hilbert series data allows us to
predict the codimension of the expected general member; our
considerations in \S\ref{s!hilb} are largely heuristic, but they are enough to provide a
line of argument for the eventual existence proof.
This approach is sometimes called the `graded rings method';
see \cite{ABR,Bk3} or \cite[\S3.2]{Rgraded}, for example.

Applying this method to find $X,A$ having
$h^0(X,A)=3$, $h^0(X,2A)=6$ and quotient singularities
$\cB = \left\{ n \times \frac{1}3(1,1,1), m \times \frac{1}5(1,1,3) \right\}$,
for various values of $n,m\ge0$,
we have candidates for embedded varieties as shown in Table~\ref{tab!cands}
(see \S\ref{s!hilb} for details).

\begin{table}[ht]
\renewcommand{\arraystretch}{2}
\setlength\tabcolsep{5pt}
\footnotesize
\[
\begin{tabular}{c|cccc}
${}_{\displaystyle n} \diagdown {}^{\displaystyle m}$ & 0 & 1 & 2  & 3\\
[0.2cm] \hline
0   & $\times$
& $X_{16}  \subset  \P(1^3,5,8)$    &  $ X_{6,10} \subset \P(1^3,3,5^2) $ &
\cellcolor[gray]{0.7} $X_{6^3, 8^3, 10^3} \subset \P(1^3,3^2,5^3) $  \\
1   & $X_{18}\subset\P(1^3,6,9)$
& $X_{11}  \subset  \P(1^3,3,5)$    & \cellcolor[gray]{0.9} $ X_{6^2, 8^2,10} \subset \P(1^3,3^2,5^2) $ \\
2   & $X_{12}\subset\P(1^3,3,6)$
& $X_{6,8} \subset \P(1^3,3^2,5)$     & \cellcolor[gray]{0.7} $X_{6^4, 8^4,10} \subset \P(1^3,3^3,5^2)  $ \\
3   & $X_9\subset\P(1^3,3^2)$
& \cellcolor[gray]{0.9} $X_{6^3, 8^2}\subset \P(1^3,3^3,5)$    \\
4   & $X_{6^2}\subset\P(1^3,3^3)$ 
& \cellcolor[gray]{0.7} $X_{6^6, 8^3} \subset \P(1^3,3^4,5) $ \\
5 & \cellcolor[gray]{0.9} $X_{6^5}\subset\P(1^3,3^4)$ 
& \multicolumn{3}{c}{Subscript of $X_{6^5}$ indicates $5$ defining equations of degree $6$,}    \\
6 & \cellcolor[gray]{0.7} $X_{6^9}\subset\P(1^3,3^5)$
& \multicolumn{3}{c}{and $\P^7(1,1,1,3,3,3,3,3)$ is abbreviated as $\P(1^3,3^5)$, etc.}    \\
\end{tabular}
\]
\caption{Candidate Calabi--Yau 3-folds with
$\cB = \left\{ n \times \frac{1}3(1,1,1), m \times \frac{1}5(1,1,3) \right\}$ for various
values of $n, m \ge0$.
(The case $n=m=0$ does not arise. Higher values of $n,m$ lead to candidates
in codimension $\ge5$.)
Those in codimension~3 are shaded light grey, those
in codimension~4 in darker grey. \label{tab!cands}}
\end{table}

Assuming for a moment that families of Calabi--Yau 3-folds exist as in Table~\ref{tab!cands},
we may treat the table as a small portion of the {\em web of Calabi--Yaus}, as
described in \cite{CGGK,GH,Prim} following Reid's fantasy conjecture in \cite{R2}.
Indeed, denoting a deformation family in Table~\ref{tab!cands} by
$\cF^{(n,m)}$ and general member of $\cF^{(n,m)}$ by $X^{(n,m)}$,
making the {\em Gorenstein projection} of \cite{PR} from an orbifold $\frac{1}3(1,1,1)$ point
$P\in X^{(n,m)}$ factorises as
\begin{displaymath}
\xymatrix{
 &  \widetilde{X}  \ar[dl] \ar[dr]  & & \\
X^{(n,m)} &&  \quad\overline{X}^{(n-1,m)}\ , &
}
\end{displaymath}
where $X^{(n,m)}\longleftarrow\widetilde{X}$ is the crepant resolution of $P\in X^{(n,m)}$ and
$\widetilde{X}\longrightarrow\overline{X}^{(n-1,m)}$ is the contraction of finitely
many rational curves to ordinary double points
on a special member of $\cF^{(n-1,m)}$.
We check that $\overline{X}^{(n-1,m)}$ admits a deformation to an
orbifold member of $\cF^{(n-1,m)}$ (see Lemma~\ref{lem!smoothY} below), and,
following \cite{Prim}, we think of this as a connection between the two families
which we denote symbolically by $\cF^{(n,m)}\rightarrow\cF^{(n-1,m)}$.
(In the cited papers, the `web of Calabi--Yaus' is a graph whose vertices
correspond to deformation families of {\em nonsingular} Calabi--Yaus;
but nonsingularity is a minor
point in our context, since  $X_{n,m}$ is $\Q$-factorial with unique crepant resolution
that factorises through $\widetilde{X}\rightarrow X^{(n,m)}$.)

Similarly, projection from an orbifold $\frac{1}5(1,1,3)$ point of $X^{(n,m)}$
describes an analogous connection $\cF^{(n,m)}\rightarrow\cF^{(n+1,m-1)}$.
Therefore the graph of connections between the families $\cF^{(n,m)}$ of
Table~\ref{tab!cands} is connected, in accordance with \cite[Conjecture~0.1]{Prim}.

In fact, we prove that these all entries of Table~\ref{tab!cands}
exist as proposed, and that in codimension~4
they arise in multiple deformation families.

\begin{thm}
\label{main theorem}
The candidates in Table~\ref{tab!cands} all exist as quasismooth, projectively
normal Calabi--Yau 3-folds in the proposed ambient spaces, nonsingular away from
quotient singularities $\cB=\left\{ n \times \frac{1}3(1,1,1), m \times \frac{1}5(1,1,3) \right\}$.
Moreover, each of those in codimension~4
admit more than one deformation component as follows: there are at least
\begin{enumerate}[(a)]
\setlength{\parskip}{0pt}
\item
2 deformation families of
$X_{6^9} \subset \P^7(1^3,3^5)$,
$\cB=\left\{ 6\times\frac{1}3(1,1,1) \right\}$.
\item
\label{thm!part}
2 deformation families of
$X_{6^6,8^3} \subset \P^7(1^3,3^4,5)$,
$\cB=\left\{4 \times \frac{1}{3}(1,1,1), \frac{1}{5}(1,1,3)\right\}$.
\item
\label{thm!partc}
3 deformation families of
$X_{6^4,8^4,10} \subset \P^7(1^3,3^3,5^2)$,
$\cB=\left\{2 \times \frac{1}{3}(1,1,1), 2\times\frac{1}{5}(1,1,3)\right\}$.
\item
\label{case-d}
2 deformation families of
$X_{6^3,8^3,10^3} \subset \P^7(1^3,3^2,5^3)$,
$\cB=\left\{3\times\frac{1}{5}(1,1,3)\right\}$.
\end{enumerate}
\end{thm}

The Hilbert series in case (\ref{case-d}) of Theorem~\ref{main theorem}
can be realised by at least one more
deformation family of Calabi--Yau orbifolds $X\subset\P(1^3,3^2,5^3)$,
but whose general member has quotient singularities
$\left\{\frac{1}3(1,1,1), \frac{1}3(2,2,2),3\times\frac{1}5(1,1,3)\right\}$;
we do not consider this family here, and its connection into the web does not
work in the same way, but see \cite{BCG} for this construction.

Our approach is by the Tom and Jerry ansatz of \cite[\S2]{TJ},
which we review in \S\ref{s!TJ}.
We give an elementary model example in \S\ref{s!example},
sketching the proof of Theorem~\ref{th!main1}, which is given in detail in~\S\ref{sec!pf1}.
In \S\S\ref{s!codim3}--\ref{s!codim4}
we construct Calabi--Yau 3-folds in codimensions~3 and~4.
Nonsingularity analysis is the crucial central point.
Bertini's theorem does not apply directly, and in any case in the presence of
nontrivial weightings linear systems tend to have large base loci;
compare Examples~3.1--3 in \cite{BKZ}.
One approach would be to write down random examples and test them by computer,
and when done on a large scale this may be the most
effective way; compare \cite{TJ}, where a few hundred cases are
settled systematically by computer.
The significant early examples of Calabi--Yau 3-folds in codimensions~3
and~4 given by Tonoli \cite{tonoli} and Bertin \cite{Ber} 
all use computer algebra as a crucial tool for checking nonsingularity.
Our aim is to give purely theoretical arguments, and in fact
we do not use computer algebra at all.

\subsection{A Hilbert scheme with two Calabi--Yau components}
\label{s!example}

To prove Theorem~\ref{th!main1}, we construct 
two polarised orbifolds $X,A$ and $X',A'$ with
the same Hilbert series but different topological invariants.
These serve as general members of two different components
of a common Hilbert scheme.
They arise as follows; \S\ref{sec!pf1} gives the details, working with
explicit equations.

In each case we construct a degree~$(3,3)$ complete intersection $Z=Z_{3,3}\subset \P^5$
that contains the union $\D\cup \E$ of two linearly-embedded copies of $\P^2\subset\P^5$.
The two constructions differ only in that $\D$ and $\E$ are chosen to be
disjoint in one case and to meet transversely in a single point in the other.
In either case, the general such $Z$  has 12 nodes on each of $\D$ and $\E$
and is nonsingular elsewhere; when $\D$ and $\E$ intersect, one of the
nodes of $Z$ lies at the intersection point, so there are only 23 nodes in total
on $Z$ in this case, rather than~24.

Blowing up each of $-\D$ and $-\E$ in $Z$ gives a small resolution
of singularities $\widetilde{Z}\longrightarrow Z$, and the birational
transforms of the planes $\D\cup \E \subset \widetilde{Z}$, now disjoint
in either case, can
be contracted to a pair of quotient singularities $\frac{1}3(1,1,1)$
on a Calabi--Yau orbifold $X$.
The polarising divisor on $Z$ lifts to one on $\widetilde{Z}$, still
meeting each of $\D$ and $\E$ in a line, and then descends to $X$.

The general nonsingular $Z_{3,3}$ has Euler characteristic
$\chi_{\mathrm{top}}=c_3=-16\times 9 = -144$.
In either situation, $\widetilde{Z}$ arises from $Z$ as a conifold transition,
passing through the degeneration to nodes and then their small resolution,
so has Euler characteristic $-144$ plus twice the number of nodes
appearing on $\D\cup \E$; thus the Euler characteristic is
either $-96$ or $-98$, depending on the case.
This distinguishes the topology of the resulting orbifolds $X$, and
so they lie in different deformation families.

In \S\ref{sec!pf1},
we show that the polarisation on $X$ embeds it in weighted projective space
$X\subset \P(1^6,3^2)$ in either case.
We calculate the two different cases separately, but first recall
Tom and Jerry from~\cite{TJ}.

\subsection{Review of Tom and Jerry}
\label{s!TJ}

The unprojection theorem \cite[Theorem 1.5]{PR} takes as input a Gorenstein
scheme $Y$ and a codimension~1 Gorenstein subscheme $D\subset Y$.
It returns a new Gorenstein scheme $X$ whose coordinate ring is an extension
of that of $Y$ by a single new generator.
The theorem is stated locally, but \cite[\S2.4]{PR} gives the graded version
we use here. The relationship between $Y$ and $X$ depends on the precise
situation, but here, as in \cite{TJ} and \cite{CPR}, it has a straightforward
birational description; see the proof of Proposition~\ref{prop!E} below.

Section~2.3 of \cite{TJ} defines two ways to set up the unprojection
data $D\subset Y$ in the case
that $Y\subset w\P^N$ is a codimension~3 scheme minimally defined by the maximal
Pfaffians of an antisymmetric $5\times 5$ matrix $M = (m_{ab})$ of forms on~$w\P^N$
and $D\subset w\P^N$ is a codimension~4 complete intersection inside $w\P^N$.
These two are called {\em Tom} and {\em Jerry}, and are respectively denoted and defined by:
\begin{description}
\item[Tom$_i$] for any $1\le i \le5$:
the entry $m_{ab}$ of $M$ must lie in $I_D$ whenever neither $a$ nor $b$ equals $i$.
In other words, all entries of $M$ lie in $I_D$ except possibly those in row and column~$i$.
\item[Jer$_{ij}$] for any $1\le i < j\le5$:
the entry $m_{ab}$ of $M$ must lie in $I_D$ whenever either $a$ or $b$ equals $i$ or $j$.
In other words, all entries of both rows and columns $i$ and $j$ of $M$ lie in~$I_D$,
while other entries are unconstrained.
\end{description}
Papadakis \cite{P} gives formulas for the unprojection $X$, but we
do not need these here.

In practice, it is often easy to show that some Tom or Jerry setup cannot work:
the degrees of entries in the matrix, combined with the row--column symmetries
one can use to simplify entries, often mean that without loss of generality some
entries can be assumed to be zero; that often implies that $Y$ is so singular
that $X$ could not be quasismooth.
We do not repeat this analysis here, since it follows \cite[\S5]{TJ} closely.

The harder part is to prove that some cases of $D\subset Y$ do
work to give, say, quasi\-smooth~$X$.
We use unprojection arguments from codimension~2, where Bertini's theorem
applies, for this part.
So in practice our approach is in three steps:
\begin{description}
\item[Step 1:]
Analyse the numerics of Hilbert series to generate candidates
for Calabi--Yau 3-folds in codimensions up to~4. (See~\S\ref{s!hilb}.)
\item[Step 2:]
Use the Tom and Jerry ansatz for candidates in codimension~4
to identify those cases that may plausibly exist as unprojections
of some $D\subset Y$ in codimension~3.
\item[Step 3:]
Construct varieties $Z$ in  codimension~2 that contain two divisors
$D\cup E\subset Z$ for which the unprojection of $E\subset Z$
realises $D\subset Y$ in codimension~3 for each case proposed
in Step 2.
\end{description}
The statements of our results come from the heuristic arguments of
Steps~1 and~2, and so are largely invisible in the discussion,
while the proofs are complete and rigorous applications
of Step~3, with various configurations of $D\cup E\subset Z$
introduced out of the blue because they work.

\subsection{Review of Hilbert series heuristics}
\label{s!hilb}

A polarised variety $X,A$ has a Hilbert series $P_X = P_{X,A}(t)$ defined by
\[
P_X = \sum_{m\in\N} h^0(X,mA) t^m,
\]
the Hilbert series of its homogeneous coordinate ring $R(X,A)$.
The orbifold Riemann--Roch formula computes $P_X$ from initial data
$P_1 = h^0(X,A)$, $P_2 = h^0(X,2A)$ and the basket $\cB$ of singularities.
We use the formula of Buckley--Reid--Zhou \cite[Theorem~1.3]{BRZ}, which applies
in any case where $K_X$ is a multiple of $A$ (in our case the zero multiple)
and $\cB$ is a collection of isolated cyclic quotient singularities.

It is convenient to be able to recognise when a series $P(t)$
is plausibly the Hilbert series of some $X,A$ in codimensions 3 or~4.
We mainly encounter the $5\times 5$ Pfaffian case
in codimension~3 and the $9\times 16$ resolution shape in codimension~4.
When $X \subset \P(a_0, \dots, a_n)$ is a variety in codimension~3
defined by the Pfaffians of a skew-symmetric $5 \times 5$-matrix $M$ with
homogeneous polynomials as entries, its  Hilbert series has the form
\begin{displaymath}
P_X(t) =
\frac{1 - t^{d_1} - \dots - t^{d_5} + t^{e_5} + \dots + t^{e_1}-t^k}{\prod_{i=0}^{n}(1-t^{a_i})}
\end{displaymath}
where the $d_i = \deg \Pf_i(M)$ are the degrees of the Pfaffians,
and $k = 2\sum \deg M_{ii}$ is the adjunction number; see \cite{CR}.
The exponents $e_i=k-d_i$ are the degrees of the syzygies.

If $X \subset \P(a_0,\dots,a_n)$ is in codimension~4
with $9 \times 16$ resolution shape, then we have
\begin{displaymath}
P_X(t) =
\frac{1 - t^{d_1} - \dots - t^{d_9} + t^{e_1} + \dots + t^{e_{16}} - t^{c_1} - \dots - t^{c_9}
+ t^{k}}{\prod_{i=0}^{n} (1-t^{a_i})}
\end{displaymath}
where again the $d_i$ are the degrees of the equations, the $e_i$ are the
degrees of first syzygies, $c_i$ the degrees of second syzygies and $k$
the adjunction number.
Typically of course the $e_i$ tend to be bigger than $d_j$, and $c_i$
tend to be bigger than $e_j$, so quite generally, as a rough guide, we interpret the number
of sign changes in the coefficients of the numerator of $P_X$ as the codimension of $X$
in $\P(a_0,\dots,a_n)$; but note the example of
Theorem~\ref{main theorem}(\ref{thm!partc}), where
\[
P_{X}(t) =  \frac{1 - 4t^6 - 4t^8 + 4t^9 - t^{10} + 8t^{11} - t^{12} + 4t^{13} - 4t^{14} - 4t^{16} + t^{22}}{(1-t)^3(1-t^3)^3(1-t^5)^2}
\]
requires a little tolerance for overlapping syzygy degrees in this rule of thumb.

In our search for varieties, then, our starting point is to fix two integers
$P_1$ and $P_2$ (which will be $h^0(X,A)$ and $h^0(X,2A)$ if we are
successful in constructing $X,A$) and the basket $\cB$ of singularities
(which we hope will be the polarised singularities of $X,A$, if successful).
Applying Riemann--Roch with these values gives
a rational function which we attempt to express in a form, such
as those above, to suggest an embedding of $X$ into some weighted projective space.
We illustrate this standard technique with an example.

\begin{eg}
Let $P_1 = 3$, $P_2 = 6$, and $\cB = \left\{4 \times \frac{1}{3}(1,1,1), \frac{1}{5}(1,1,3) \right\}$.
Orbifold Riemann--Roch \cite{BRZ} computes a function that we denote by~$P$:
\begin{displaymath}
P = \frac{1 -t - t^3 + t^4}{(1-t)^4} + 4 \times \frac{t^3}{(1-t)^3(1-t)} + \frac{t^5 + t^3}{(1-t)^3(1-t)}
\end{displaymath}
To express $P$ in a recognisable `free resolution form' as above, we first expand
as a power series,
\begin{displaymath}
P = 1 + 3t + 6t^2 + 14t^3 + 27t^4 + 46t^5 + \cdots,
\end{displaymath}
and then multiply by factors $1-t^{a_i}$ for weights $a_i$ that we
guess inductively.
In this case, $P_1=3$, the first nontrivial coefficient of a power of $t$,
so we expect three variables of weight~1, and
correspondingly multiply both sides of the above equation by $(1-t)^3$ to get
\begin{displaymath}
(1-t)^3P = 1 + 4t^3 + t^5 + 4t^6 + t^8 + 4t^9 + \cdots.
\end{displaymath}
The next nontrivial power of $t$ is $4t^3$, so we multiply by $(1-t^3)^4$
to eliminate this term:
\begin{displaymath}
(1-t)^3(1-t^3)^4P = 1 + t^5 - 6t^6 - 3t^8 + \cdots.
\end{displaymath}
Continuing, we multiply by $(1-t^5)$ to get
\begin{displaymath}
(1-t)^3 (1-t^3)^4 (1-t^5) P = 1-6t^6 - 3t^8 + 8t^9 + 8 t^{11} - 3 t^{12} - 6t^{14} + t^{20}
\end{displaymath}
and conclude that
\begin{displaymath}
 P = \frac{1-6t^6 - 3t^8 + 8t^9 + 8 t^{11} - 3 t^{12} - 6t^{14} + t^{20}}{(1-t)^3 (1-t^3)^4 (1-t^5)}.
\end{displaymath}
(Comparing with the original expression shows that the tail of
the power series cancels.)
The number of sign changes in the numerator of $P$ is four, which 
suggests a codimension four 3-fold
$X \subset \P(1^3,3^4,5)$ with Hilbert series $P_X=P$; this is case (b)
of Theorem~\ref{main theorem}.
In \S\ref{s!codim4} we prove
that there are in fact two deformation families of such a quasismooth~$X$.
\end{eg}

Analogous calculations with different
$\cB = \left\{ n \times \frac{1}{3}(1,1,1), m \times \frac{1}{5}(1,1,3) \right\}$
give the candidates of Table~\ref{tab!cands}.
There are a couple of tricks we employ beyond the basic game
to end up with candidates that can be realised in practice.
The singularities of $\cB$ should arise as restrictions of those
of the ambient weighted projective space; for example,
if $\cB$ includes $\frac{1}5(1,1,3)$, then the weights $a_i$
should include some that are divisible by~5, and at least
one that is congruent to~3 modulo~5. And a $\frac{1}5(1,1,3)$
point would lead to a $\frac{1}3(1,1,1)$ point after projection, so unless
the degree $A^3\le 1/15$ (so there could be no projection)
there should be some weights divisible by~3 to accommodate this.
If the power series process does not provide such weights, we must experiment
a little by hand.
Although this is somewhat heuristic, it is justified in the end by the
results of \S\S\ref{s!codim3}--\ref{s!codim4}.

\section{The proof of Theorem~\ref{th!main1}}
\label{sec!pf1}

\subsection{Disjoint planes, parallel unprojection and Jerry}
\label{s!jerry}

Consider $\P^5$ with homogeneous coordinates $x,y,z,u,v,w$.
Let $\D=(u=v=w=0)$ and $\E=(x=y=z=0)$.
Any 3-fold $Z=(f=g=0)$ containing $\D\cup \E$ is defined by two polynomials
\begin{eqnarray*}
 f &=& X_2u + Y_2v + Z_2w + U_2x + V_2y + W_2z \\
 g &=& X_2'u + Y_2'v + Z_2'w + U_2'x + V_2'y + W_2'z,
\end{eqnarray*}
where the forms $X_2=X_2(x,y,z)$, \dots, $Z_2'(x,y,z)$
are quadrics in $x,y,z$ and the forms
$U_2=U_2(u,v,w)$, \dots, $W_2'(u,v,w)$ are
quadrics in $u,v,w$.
Since all these forms can be chosen freely,
the general such $Z=(f=g=0)$ is nonsingular
away from the base locus $\D\cup \E$ by Bertini's theorem.

\begin{lem}
\label{lem!Z}
In the notation above, the singular locus $\Sing(Z)$ of $Z$ consists of exactly 12 points
on $\D$ and 12 points on $\E$, and all the singular points are ordinary nodes.
\end{lem}

\begin{proof}
Consider first the singularities that lie on $\E$.
We rewrite the equations of $Z$ as
\begin{equation}
\label{e!fg}
\begin{pmatrix}
f\\g
\end{pmatrix}
=
\begin{pmatrix}
A&B&C\\
D&E&F
\end{pmatrix}
\begin{pmatrix}
x\\y\\z
\end{pmatrix}
=
\begin{pmatrix}
0\\0
\end{pmatrix},
\end{equation}
where $A,\dots,F$ are quadric forms  in the ideal $I_\D=(u,v,w)$.
In this notation, the singularities of $Z$ on $\D$ occur where the $2\times 3$
matrix in~\eqref{e!fg} drops rank. If we choose both $A$ and $E$ to
be the zero form, this rank condition is equivalent to
\[
BD = CD = BF = 0.
\]
With this choice, the singular locus of $Z$ on $\E$ is
\begin{equation}
\label{e!singZ}
\Sing(Z) \cap \E = (B = C = 0) \cup (B = D = 0) \cup (D = F = 0).
\end{equation}
Each of the factors in this union is the intersection of two
conics in $\E=\P^2$, and so for general $B,C,D,F$ this
singular locus is 12 distinct reduced points, which are therefore necessarily nodes.
Since the singularities are nodes for this choice of $A,\dots,F$, they
are also nodes for the general member $Z$ of this family.

The analogous argument for $\D$ shows that the general member
of the family also has exactly 12 nodes on $\D$,
which completes the proof of the lemma.
\end{proof}

We refer to this choice of $A, \dots, F$ in the proof of the lemma
as a \emph{standard choice};
we use it throughout to show that a Jacobian ideal is reduced
in general, and so the singularities are nodes.

Next we make the $\E$-ample small blowup of $Z$ followed by contraction of $\E$ in one step
using the Kustin--Miller unprojection of \cite[ Theorem~1.5]{PR}; we set up the algebra
of unprojection first before proving this geometric claim in Proposition~\ref{prop!E}.

Unprojecting $\E\subset Z$ introduces a new variable $s$, the {\em unprojection variable}.
With the equations of $Z$ written as in~\eqref{e!fg}
the unprojection calculus realises $s$ as the rational form
\begin{equation}
\label{e!s}
s =
\frac{BF-CE}x =
\frac{AF-CD}{-y} =
\frac{AE-BD}z.
\end{equation}
The five equations---namely $f, g$ and with the implicit equations for $s$ above, $sx = BF-CE$
and so on---can be mounted as the maximal Pfaffians of an anti-symmetric $5\times 5$
matrix (we display only the strict upper triangle)
\begin{equation}
\label{e!j12}
\begin{pmatrix}
z & y & A & D \\
   & x & B & E \\
        && C & F \\
              &&& s
\end{pmatrix}.
\end{equation}
In the language of \S\ref{s!TJ},
the matrix \eqref{e!j12} is in Jerry format for the ideal~$I_\D$; more
precisely, it is Jer$_{45}$, since each entry of the fourth and fifth rows
and columns lies in~$I_\D$.

These five equations define a 3-fold $Y\subset \P(1^6,3)$, in coordinates $x,y,z,u,v,w,s$,
which contains the birational transform of $\D$, also denoted $\D\subset Y$,
defined by $(u=v=w=s=0)$.

\begin{prop}
\label{prop!E}
The point $P_s = (0:0:0:0:0:0:1) \in Y$ is quasismooth, and is a quotient
singularity $\frac{1}{3}(1,1,1)$. Away from $P_s$, $Y$ is smooth outside $\D$,
and has exactly 12 isolated singularities on $\D$ all of which are nodes.
\end{prop}

\begin{proof}
The three unprojection equations ($sx=BF-CE$, etc., of~\eqref{e!s}) eliminate $x,y,z$
near~$P_s$, so $Y$ is quasismooth there, and has an orbifold chart
(analytically) locally isomorphic to a neighbourhood of the origin in the
index~3 orbifold chart of the ambient weighted
projective space; locally it is just the origin in the quotient of $\C^3$
with coordinates $u,v,w$ by $\Z/3$ acting by multiplication, which
is the claim.

The unprojection map $\varphi\colon Z \dashrightarrow Y$ is defined by
\[
(x,y,z,u,v,w) \longmapsto (x,y,z,u,v,w,s)
\]
where $s$ is the rational function defined in~\eqref{e!s}.
Multiplying through by the denominators of~$s$ gives polynomial expressions
for $\varphi$,
\begin{eqnarray*}
(x,y,z,u,v,w) &\longmapsto& (x^2,xy,xz,xu,xv,xw,BF-CE) \\
 && \qquad = (xy,y^2,yz,yu,yv,yw,CD-AF) \\
 && \qquad = (xz,yz,z^2,zu,zv,zw,AE-BD),
\end{eqnarray*}
from which it follows immediately that $\varphi$ is
defined everywhere on $Z$ except possibly at the locus
\[
x=y=z=AE - BD = CD - AF = BF - CE=0.
\]
This locus is precisely the set of nodes $\{Q_1,\dots,Q_{12}\}$ on $\E$.
The birational inverse $\psi\colon Y\dashrightarrow Z$ is simply the elimination of~$s$,
and so is defined everywhere on $Y$ except at the point $P_s$.

Denote the (closure of the) preimage of a node $Q_i$ on
$\D$ under $\psi$ by $\Gamma_i$.
Obviously, $\Gamma_i \cong \P^1$, since it is the cone on $Q_i$ over $P_s$.
The restriction of $\varphi$ on $Z \setminus \E$ gives an isomorphism
\begin{equation}
\label{e!ZEYG}
Z \setminus \E \longrightarrow Y \setminus \Gamma,
\qquad
\text{where $\Gamma = \bigcup_{ i = 1}^{12}\Gamma_i$}.
\end{equation}
Note that $\D \subset Z \setminus \E$, so a neighbourhood of $\D$ is mapped isomorphically,
and its 12 nodes survive on~$Y$.
Since $Z$ is smooth away from $\D$ and $\E$ the isomorphism~\eqref{e!ZEYG}
shows that $Y$ is smooth away from $\D$ and $\Gamma$.
Furthermore, the equations defining $\varphi$ show that it induces a contraction
\[
\E \setminus \Sing(Z) \longrightarrow \{P_s\}.
\]
It remains only prove that $Y$ is quasismooth on $\Gamma\setminus \{P_s\}$.
This is elementary:
the locus $\Gamma$ is defined by the same equations~\eqref{e!singZ}
that define the nodes on~$\E$, and using a {\em standard choice} of forms $A,\dots,F$,
the Jacobian matrix of $Y$ restricted to any of the
three components of the union in~\eqref{e!singZ}
has rank three along that part of~$\Gamma$
(c.f.~a similar elementary calculation for a harder case in \S\ref{s!codim3} below).
\end{proof}

Next we unproject $\D\subset Y$.
By  \cite[ Theorem~1.5]{PR},
this introduces another unprojection variable of degree~3,
and constructs some $X\subset \P(1^6,3^2)$, embedded in codimension~4.
The whole picture is this:
\begin{displaymath}
\xymatrix@C=-3pt{
 &   &  X   & \subset & \hspace{0.1cm}\P(1^6,3^2)\\
  &  \D \hspace{0.1cm}\subset &  Y \ar@{-->}[u]     & \subset & \hspace{0cm}\P(1^6,3)\\
\D\; \cup & \E \hspace{0.1cm}\subset &  Z \ar@{-->}[u] & \subset & \hspace{-0.9cm}\P^5 }
\end{displaymath}
and we prove next that $X$ is an orbifold birational to~$Z$.

\begin{thm}
\label{thm!unprojto4}
In the notation above,
$X$ is a quasismooth 3-fold with exactly two $\frac{1}{3}(1,1,1)$ quotient singularities.
\end{thm}

\begin{proof}
The variety $X$ is defined to be the unprojection of $\D \subset Y$.
The crucial point is to show that $X$ is quasismooth.
As well as the two unprojections used so far, we consider
the unprojection of $\D \subset Z$ to give another variety $\overline{Y}$
in codimension 3 that contains an isomorphic copy $\E \subset \overline{Y}$.
\begin{displaymath}
\xymatrix@C=-3pt{
                                                              &  X \hspace{0.2cm}      & \\
   \E \subset \overline{Y} \hspace{0.9cm} &         & \hspace{0.8cm} \ar@{-->}[ul]_{\varphi_2} \hspace{0.2cm}Y  \supset  \D\\
   \D,\E \subset     \hspace{-0.5cm}&  Z \ar@{-->}[ur]_{\varphi} \ar@{-->}[ul]^{\bar{\varphi}} &
}
\end{displaymath}
By Proposition~\ref{prop!E}, $\varphi$ is an isomorphism in a neighborhood of~$\D$,
and the analogous argument shows that
$\bar{\varphi}$ is an isomorphism in a neighborhood of $\E$
and $\overline{Y} \setminus \E$ is quasismooth.

By~\eqref{e!ZEYG} there is an isomorphism $Z \setminus \E \cong Y \setminus \Gamma$,
where $ \Gamma = \bigcup_{i=1}^{12} \Gamma_i$ is the union of exceptional rational
curves blown up from the nodes lying on~$\E$.
Therefore, the unprojection of $\D \subset Z \setminus \E$ is isomorphic to 
the unprojection of $\displaystyle \D \subset Y \setminus \Gamma$.
\begin{displaymath}
\xymatrix@C=-3pt{
  &  X \setminus \Gamma           & \\
   \overline{Y} \setminus \E \hspace{0.4cm} &     &\hspace{0.2cm}\ar@{-->}[ul] Y \setminus \Gamma \supset \D\\
   \D \subset \hspace{-0.5cm}    &  Z \setminus \E \ar@{-->}[ur] \ar@{-->}[ul] &
}
\end{displaymath}
(We have identified $\Gamma$ with its isomorphic image in $X$ in the diagram.)
Consequently, $X \setminus \Gamma \cong \overline{Y} \setminus \E$,
and so $X \setminus \Gamma$ is quasismooth. Since $\varphi_2$ is an
isomorphism in a neighborhood of $\Gamma$ it follows that $X$ is quasismooth.
Moreover, $\D$ and $\E\subset Z$
are contracted to two $\frac{1}3(1,1,1)$ quotient singularities on~$X$, 
since they are on $\overline{Y}$ and $Y$, respectively.
\end{proof}

\subsection{Intersecting planes, serial unprojection and Tom}

For the second calculation, consider instead a pair of
planes $\D=(x=u=v=0)$ and $\E=(x=y=z=0)$
which meet transversely in the coordinate point $P_w=(0:0:0:0:0:1)\in\P^5$.
As before, the general complete intersection $Z'=Z_{3,3}'=(f=g=0)$ containing $\D\cup \E$
is also nonsingular outside 12 nodes on each of $\D$ and $\E$, although
this time the point $P_w\in Z'$ is necessarily one of those nodes, and so
there are only 23 in total.

Unprojecting $\E\subset Z'$ first, as before, the equations $f,g$ can be presented
as in \eqref{e!fg}, but this time under the conditions that $B,C,E,F\in (u,v)$,
so that both $f$ and $g$ lie in $I_{\D\cup \E}=(x,yu,yv,zu,zv)$.
The unprojection variable $s$ satisfies the same equations~\eqref{e!s} as before,
and the five equations of the unprojection $Y'\subset\P(1^6,3)$ can again be presented
as the maximal Pfaffians of the matrix \eqref{e!j12}.

But now, in contrast to \S\ref{s!jerry},
the matrix \eqref{e!j12} for $Y'$ is in Tom format for the ideal~$I_\D$,
in fact Tom$_1$, since each entry outside the first row and column lies in~$I_\D$.

Using the same arguments as Lemma~\ref{lem!Z} and Proposition~\ref{prop!E},
general choices for the forms $A,\dots,F$ defining $Z'$,
subject to the condition $B,C,E,F\in (u,v)$, results in a 3-fold $Y'$ containing an isomorphic
copy $\D\subset Y'$ that is quasismooth away from $\D$ and has only nodes on $\D$ itself.
This unprojection of $\E\subset Z$ resolves all 12 nodes lying on $\E$,
including the one that also lies on $\D$, and so $Y'$ has only 11 nodes.

Again, unprojecting $\D\subset Y'$ constructs a quasismooth variety
$X'\subset \P(1^6,3^2)$ with two $\frac{1}3(1,1,1)$ quotient singularities.


Thus we have constructed two 3-fold orbifolds $X, X'\subset\P(1^6,3^2)$.

\begin{thm}
The varieties $X$ and  $X'$ are quasismooth Calabi--Yau 3-folds with
the same Hilbert series but with different topological Euler characteristic.
In particular, $X$ and $X'$ lie in different components of Calabi--Yau 3-fold
orbifolds in the same Hilbert scheme.
\end{thm}

\begin{proof}
By construction, the two varieties have the same quotient singularities $2\times\frac{1}3(1,1,1)$.
The hypersurfaces $Z$ and $Z'$ have the same Hilbert series, and
by \cite{PR} their Hilbert series are related to that by
\[
P_X = P_Z + \frac{2t^3}{(1-t)^3(1-t^3)} = P_{X'}.
\]

The initial variety $Z\subset\P^5$ is a Calabi--Yau 3-fold.
Unprojection of the planes preserves this property.
Indeed all singularities are rational, so $h^1(X,\cO_X)=h^1(X',\cO_{X'})=0$.
(Alternatively one can use the unprojection theorem~\cite[Theorem~5.1]{PR}
directly: it proves that the homogeneous coordinate ring of each unprojection is
Gorenstein, and so the theorem of Goto--Watanabe \cite[5.1.9--11]{GW} implies
all intermediate cohomology vanishes; in particular $h^1(X,\cO_X)=0$,
and similarly for $X'$.)
By the proof of Proposition~\ref{prop!E},
geometrically each unprojection factorises as the crepant resolution of
$\frac{1}3(1,1,1)$ and small resolution of nodes (which is necessarily crepant),
so $K_X=0 $ and $K_{X'}=0$. Thus $X$ and $X'$ are both Calabi--Yau 3-folds.

It remains to check the claim about Euler characteristics.
We consider the final unprojection to $X$ and $X'$ from codimension~3 Calabi--Yau 3-folds
$Y$ and $Y'$.
For $X$, we unproject a nodal variety $Y$ containing a divisor $\D$ in $\text{Jer}_{45}$ format;
for $X'$, we unproject a nodal variety $Y'$ with a divisor, also denoted $\D$, in $\text{Tom}_1$ format.
In both cases, the unprojection $\varphi$ factorizes as the $\D$-ample small
resolution of nodes, $\tilde{Y}\rightarrow Y$ and $\tilde{Y}' \rightarrow Y'$, followed by the
contraction, $\tilde{Y} \rightarrow X$ and $\tilde{Y}' \rightarrow X'$, of $\D$.

Both $Y$ and $Y'$ are defined by the maximal Pfaffians of
skew $5\times 5$ matrices of forms that have degrees
\begin{equation}
\label{eq!wts}
\begin{pmatrix}
1 & 1 & 2 & 2 \\  & 1 & 2 & 2 \\  && 2 & 2 \\  &&& 3
\end{pmatrix}.
\end{equation}
\begin{lem}
\label{lem!smoothY}
Any 3-fold $\Ygen\subset\P(1^6,3)$ defined
by the maximal Pfaffians of a matrix of {\em general} forms with degrees
as indicated by \eqref{eq!wts} is quasismooth.
\end{lem}
This lemma follows at once by applying the proof of Proposition~\ref{prop!E}
to the case of general $Z_{3,3}\subset\P^5$ which contains $\E$ but
is not constrained to contain~$\D$.

Each of $Y$ and $Y'$ is a (flat) degeneration of $\Ygen$: varying entries
in the skew matrix deforms the entire free resolution of the defining ideal,
and $Y$ and $Y'$ are made by varying particular Tom or Jerry entries to lie in
the ideal $I_\E$.
All this fits in a picture:
\begin{displaymath}
\xymatrix{
 X' & \ar[l] \tilde{Y}'  \ar[d]&                                                    & \tilde{Y} \ar[d] \ar[r]     & X\\
    & \ar@{-->}[ul]^{\varphi}  Y'        & \ar@{~>}[l]  \Ygen  \ar@{~>}[r]    & Y  \ar@{-->}[ur]_{\varphi}  &
}
\end{displaymath}
where \textquotedblleft$\xymatrix@1{\ar@{~>}[r]&}$\textquotedblright denotes degeneration.
The passage from $\Ygen$ to $\tilde{Y}$ or $\tilde{Y}'$ is a conifold transition.
By Clemens' calculation \cite{Clemens,R2}, the Euler characteristic of $\tilde{Y}$ is related
to that of $\Ygen$ by
\begin{displaymath}
\chi(\Ygen) = \chi(\tilde{Y}) + 2 \cdot 12,
\end{displaymath}
where 12 is the number of nodes on $\D$ in the $\text{Jer}_{45}$ case.
Similarly $\chi(\Ygen) = \chi(\tilde{Y}') + 2 \cdot 11$,
where 11 is the number of nodes on $\E$ in the $\text{Tom}_1$ case.
So
\[
\chi(\tilde{Y}') = \chi(\tilde{Y}) + 2.
\]
The maps $\tilde{Y} \rightarrow X$ and $\tilde{Y}' \rightarrow X'$ are
crepant blowups of the same singularity $\frac{1}{3}(1,1,1)$,
so the change in the Euler characteristic from $\tilde{Y}$ to $X$ and from $\tilde{Y}'$ to $X'$
is the same in either case.
Therefore, $\chi({X}') = \chi({X}) + 2$, and $X$ and $X'$ have different Euler characteristics.
\end{proof}

The Hilbert series $P_X=P_{X'}$ can be expanded as
\[
P_X = P_{X'} =
\frac{1 - 2t^3 - 6t^4 - t^6 + 6t^5 + 4t^6 + 6t^7 - t^6 - 6t^8 - 2t^9 + t^{12}}{(1-t)^6(1-t^3)^2},
\]
presenting the numerator to emphasise the 9 equations and 16 syzygies,
or $9\times 16$ resolution, even though one could cancel some $t^6$ terms.
The two varieties $X$ and $X'$ that we construct are both defined minimally
by 2 cubics, 6 quadrics and a single sextic in the variables of $\P(1^6,3^2)$.
Our approach uses the unprojection theorem \cite[Theorem~1.5]{PR}
merely to show that $X$ exists in the ambient space we claim; it also
determines the degrees of the equations, which is how we see the degree~6
equation that is otherwise masked in the Hilbert numerator.
We could instead apply \cite[\S5.7]{P} to write down the equations of $X$,
but we do not need them here.

\begin{cor}
\label{cor!batyrev}
$X$ and $X'$ are not birational.
\end{cor}

This follows from Batyrev's theorem~\cite[Theorem~1.1]{B}:
smooth Calabi--Yau 3-folds that are birational have the same Betti numbers.
Indeed, if $X$ and $X'$ were birational, then, in the notation of the proof, $\tilde{Y}$
and $\tilde{Y}'$ would be birational too. But these are smooth Calabi--Yau 3-folds,
and the proof showed that their Euler characteristics differ, and so their Betti
numbers cannot be equal.


\section{Calabi--Yau 3-folds in codimension 3}
\label{s!codim3}

We proceed with the proof of Theorem~\ref{main theorem}.
The complete intersection Calabi--Yau 3-folds are all well known and follow
at once from Bertini's theorem.
We show that the general $5\times 5$ codimension~3 variety
$Y\subset\P(1^3,3^3,5)$ defined with syzygy degrees
\[
\begin{pmatrix}
1 & 3 & 3 & 3 \\
 & 3 & 3 & 3 \\
   && 5 & 5 \\
     &&& 5
\end{pmatrix}
\]
is quasismooth; such $Y$ has quotient singularities
$\left\{3\times\frac{1}3(1,1,1), \frac{1}5(1,1,3)\right\}$,
which is the case $n=3$, $m=1$ of Table~\ref{tab!cands}.

The starting point is the general codimension~2 complete intersection
\[
\E \subset Z_{6,6} \subset \P(1^3,3^3),
\]
where $\E=(z = u = v = 0) = \P(1,1,3)$, in coordinates $x,y,z,u,v,w$ on $\P(1^3,3^3)$.

\begin{lem}
In the notation above, the non-quasismooth locus of $Z$ consists of exactly
13 points, and these are all ordinary nodes lying on $\E$.
\end{lem}

The proof follows that of Lemma~\ref{lem!Z}.
We write the equations of $Z$ as
\[
\begin{pmatrix}
A_5&B_3&C_3\\
D_5&E_3&F_3
\end{pmatrix}
\begin{pmatrix}
x\\y\\z
\end{pmatrix}
=
\begin{pmatrix}
0\\0
\end{pmatrix},
\]
for general forms $A,\dots,F$ of the indicated degrees.
Bertini's theorem shows that $Z$ is quasismooth away from~$\E$.
The non-quasismooth locus of $Z$ along $\E$, denoted $\Sigma$,
occurs where the $2\times 3$ matrix of forms drops rank, so has Hilbert series
\begin{eqnarray*}
P_{\Sigma | \E} &=& \frac{1 - t^6 - 2t^8 + 2t^{11}}{(1-t)^2(1-t^3)} \\
 &=& \frac{1+ t + t^2 + 2t^3 + 2t^4 + 2t^5 + 2t^6 + 2t^7}{1-t},
\end{eqnarray*}
which is that of a 0-dimensional scheme of length 13.
The standard choice $A=E=0$ and $B, C, D, F$ general gives cases with
\begin{equation}
\label{e!sigma}
\Sigma = (B=C=0) \cup (B = D = 0) \cup (D = F = 0) \subset \E = \P(1,1,3)
\end{equation}
reduced and disjoint from the orbifold locus of $Z$; so the singularities 
of such $Z$ are ordinary nodes,
and so the same holds for general $Z$ and the
lemma is proved.

Unprojecting $\E$ gives a new 3-fold~$Y$:
\begin{displaymath}
\xymatrix@C=-3pt{
                 &              &  Y  \ar@/_/@{-->}[d]_{\psi}       & \subset & \hspace{0.1cm} \P(1^3,3^3,5)\\
                 & \E \subset &  Z \ar@{-->}[u]_{\varphi} & \subset & \hspace{-0.3cm} \P(1^3,3^3)
}
\end{displaymath}
with unprojection variable $s$ of degree 5.
This variety $Y$ is defined by three unprojection
equations and the defining equations $f, g$ of $Z$:
\begin{displaymath}
sz  -  BF + CE = su + AF - DC = sv - AE + BD = f = g =  0.
\end{displaymath}
We follow the proof of Proposition~\ref{prop!E}.
The preimage in $Y$ of the nodes of $Z$ is a bouquet of rational curves
$\Gamma = \bigcup_{i=1}^{13} \Gamma_i$, the cone over $\Sigma$ with vertex
the coordinate point $P_s$.
The restriction of $\varphi$ on $Z \setminus \E$ is an isomorphism
\begin{displaymath}
\xymatrix{ Z \setminus \E \ar[r] & Y \setminus \displaystyle  \Gamma },
\end{displaymath}
so $Y$ is quasismooth outside $\Gamma$, and
furthermore the equations of $\varphi$ show at once that it gives a contraction
$\xymatrix{ \E \setminus \Sing(Y) \ar[r] & \{ P_s \} }$.
It remains to show that $Y$ is quasismooth on $\Gamma$.
Using the {\em standard choice} $A = E = 0$, we calculate
the Jacobian of the five defining equations \eqref{e!sigma} restricted to
the three pieces of the union in turn.
First, consider the piece $( z = u = v = B = C = 0 )$.
The Jacobian restricted to the rational curves $\Gamma_i$ of this piece is
\begin{displaymath}
\begin{pmatrix}
-B_x F & -B_y F & s - B_z F & -B_u F & -B_v F & -B_w F & 0 \\
-DC_x  & -DC_y  & -DC_z     & s-DC_u & -DC_v  & -DC_w  & 0 \\
B_x D  & B_y D  & B_z D     & B_u D  & s+B_vD & B_w D  & 0 \\
0      & 0      & 0         & 0      & 0      & 0      & 0 \\
0      & 0      & D         & 0      & F      & 0      & 0
\end{pmatrix}.
\end{displaymath}
Consider the submatrix
\begin{displaymath}
\begin{pmatrix}
-DC_x  & -DC_y        & -DC_v  & -DC_w   \\
B_x D  & B_y D        & s+B_vD & B_w D   \\
0      & 0            & F      & 0
\end{pmatrix}.
\end{displaymath}
with entries in $x,y,w$ and $s$.
In general, $(B = C = 0)$ defines 3 distinct points in $\E=\P(1,1,3)$---a
nonsingular zero-dimensional subscheme of~$\E$. The polynomial
$D$ is not zero there, so columns 1, 2, 4 of the top two rows of this submatrix has rank 2.
The general $F$ is not zero there either, so the submatrix, and so in particular
the Jacobian matrix itself, has rank 3 on this piece of the union that makes up $\Gamma$.
The argument is the similar for the other two pieces.


\section{Calabi--Yau 3-folds in codimension 4}
\label{s!codim4}

We prove Theorem~\ref{main theorem}(\ref{thm!part}); that is, we construct members
of two deformation families of Calabi--Yau 3-folds
\[
X_{6^6,8^3}\subset\P(1^3,3^4,5)
\quad
\text{with quotient singularities }
\textstyle{4\times\frac{1}3(1,1,1), \frac{1}5(1,1,3)}.
\]
The starting point is a $(6,6)$ complete intersection $Z_{6,6} \subset \P^5(1^3,3^3)$
containing divisors $\D\cong\P^2$ and $\E\cong\P(1,1,3)$.
As in the model case, we build a variety $\D\subset Y$ by unprojecting $\E$ in $Z$,
and then construct $X$ by unprojecting $\D$ in $Y$, but
unlike the model case these divisors cannot both be chosen to be coordinate hyperplanes.
The choice of $\E$ is delicate:
after unprojection we want a variety $Y$ that contains $\D$ in
$\text{Tom}_1$ format or, later on, in $\text{Jer}_{45}$ format, and this
constrains $\E$.

\subsection{Construction of $\text{Tom}_1$}
\label{s!tom1}

Fix homogeneous coordinates $x,y,z,t,u,v$ on $\P(1^3,3^3)$
and define $\D=\P^2$ by $I_\D = (t, u, v)$.
Define $\E$ by setting $\alpha  =  x$, $\beta   =  yz(y-z) + u$, and $\varepsilon = t$ and taking
$I_\E = (\alpha, \beta, \varepsilon)$. It follows that $I_{\D \cup \E } = (t, ux, u \beta,
vx, v \beta)$ and
\begin{displaymath}
\D \cap \E = \{(0:1:0:0:0:0), (0:0:1:0:0:0), (0:1:1:0:0:0)\}.
\end{displaymath}
Let $Z=Z_{6,6}$ be a general $(6,6)$ complete intersection that contains $\D\cup\E$.

\begin{lem}
In the notation above, the singular locus $\Sing(Z)$ of $Z$ consists of
exactly 27 points on $\D$ and 13 points on $\E$,
and all singular points are ordinary nodes.
In each case, 3 of the nodes are at $\D\cap\E$, so that $Z$ has 37 nodes in total.
\end{lem}

We follow the proof of Lemma~\ref{lem!Z}, using the same standard choice.

\begin{proof}
First write the equations of $Z$, from the point of view of $\E\subset Z$, as
\begin{equation}
\label{e!fg66}
\begin{pmatrix}
A_5&B_3&C_3\\
D_5&E_3&F_3
\end{pmatrix}
\begin{pmatrix}
\alpha\\\beta\\\epsilon
\end{pmatrix}
=
\begin{pmatrix}
0\\0
\end{pmatrix},
\end{equation}
where $A,\dots,F$ are forms of the indicated degrees and $A, B, D, E\in I_\D$.
The singular locus of $Z$ on $\E$ is defined by the $2\times 2$ minors of the
matrix in~\eqref{e!fg66}.
Making the standard choice that both $A$ and $E$ are the zero form, we have
\begin{equation}
\Sing(Z) \cap \E = (B_3 = C_3 = 0) \cup (B_3 = D_5 = 0) \cup (D_5 = F_3 = 0).
\end{equation}
We may choose further $B'=v$ and $F'=v+h_3$, for a general cubic $h=h(x,y,z)$,
so that each factor in this disjoint union is a transverse intersection away from the
index~3 point of $\E\cong\P(1,1,3)$, of 3, 5 and 5 reduced points respectively.
(Since $B,D\in I_\D$, the three points of $\D\cap\E$ fall into the second factor of the union;
the other factors do not meet this intersection since $C$ and $F$ are not constrained.)
Thus the general $Z$ has 13 nodes on $\E$.

Now re-expressing the equations of $Z$ from the point of view of
$\D\subset Z$ as
\begin{equation}
\label{e!fg66D}
\begin{pmatrix}
A'&B'&C'\\
D'&E'&F'
\end{pmatrix}
\begin{pmatrix}
t\\u\\v
\end{pmatrix}
=
\begin{pmatrix}
0\\0
\end{pmatrix},
\end{equation}
where $A',\dots,F'$ are cubic forms and $B',C',E',F'\in I_\E$.
The singular locus of $Z$ on $\D$ is defined by the $2\times 2$ minors of the
matrix in~\eqref{e!fg66D}.
Making the standard choice $A'=E'=0$, we have
\begin{equation}
\Sing(Z) \cap \D = (B' = C' = 0) \cup (B' = D' = 0) \cup (D' = F' = 0).
\end{equation}
Since $B',C'\in I_\E$, the first factor in this union consists of 9 reduced points
in $\D=\P^2$, of which 3 are $\D\cap\E$.
The cubic $D'$ is not constrained, so the remaining two factors
provide a further 9 disjoint reduced points each,
and so for general $B',C',D',F'$ this
singular locus is 27 distinct reduced points, all necessarily nodes.
Since the singularities are nodes for this choice of $A',\dots,F'$, they
are also nodes for the general member $Z$ of this family.
\end{proof}

\paragraph{Unprojection of $\E$}
At this point, we have a general $Z_{6,6} \in \P(1^3,3^3)$ containing two divisors.
Unprojecting $\E$ gives a 3-fold $Y \subset \P(1^3,3^3,5)$ in codimension~3:
\begin{displaymath}
\xymatrix@C=-3pt{
                 &  \D          &\subset &  Y                         & \subset &\P(1^3,3^3,5)\\
              \E,  &  \D &\subset &  Z_{6,6}   \ar@{-->}[u]    & \subset & \hspace{-0.3cm}\P(1^3,3^3)\\
}
\end{displaymath}
The unprojection variable, denoted by $w$, has weight 5.
In the notation of \eqref{e!fg66}, the equations of
$Y$ are those of $Z_{6,6}$ together with three new equations:
\begin{displaymath}
w \alpha  =  BF - CE, \quad w \beta  =  - AF + CD, \quad w \varepsilon  =  AE - BD.
\end{displaymath}
As usual, these five equations can be realised as the maximal Pfaffians of a matrix,

\begin{displaymath}
\begin{pmatrix}
\alpha & \beta & C & F \\
      & \epsilon & -B & -E \\
&        & A & D \\
  && & w
\end{pmatrix}.
\end{displaymath}
By design, these are in $\text{Tom}_1$ format for $I_\D = (t,u,v,w)$.
The polynomials $A, \dots, F$ do not contain $w$, so
$P_w = \frac{1}{5}(1,1,3) \in Y$ is a quasismooth point: the coefficients of $w$ on
the left-hand side of the three unprojection equations contain independent linear terms.

\begin{prop}
$Y$ is quasismooth outside $\D$, and has in total 24 isolated singularities on $\D$,
all of which are nodes.
\end{prop}

\begin{proof}
Also here, we follow the same pattern as in the proof of Proposition 2.2, and use analogous definitions and notation.
The restriction of $\varphi$ on $Z \setminus \E$ gives an isomorphism
\begin{equation}\label{iso}
\xymatrix{ Z \setminus \E \ar[r] & Y \setminus  \bigcup_{ i = 1}^{13}\Gamma_i }.
\end{equation}
Note that $\left( \D \setminus (\D\cap\E) \right) \subset Z \setminus \E$. Since
$Z_{6,6}$ is quasismooth away from $\D\cup\E$ the isomorphism \eqref{iso}
shows that $Y$ is
quasismooth away from $\D$ and $ \bigcup_{i=1}^{13} \Gamma_i$ and that
$\D$ contains 24 nodes of~$Y$.
Furthermore, $\varphi$ induces a contraction $\xymatrix{ \E \setminus \Sing(Z) \ar[r] & \{ P_w \}
}$.
As before, the unprojection factorizes as the $\E$-ample small resolution followed
by the contraction of $\E$, and $Y$ is quasismooth along the $\Gamma_i$
by the Jacobian argument of~\S\ref{s!codim3}.
\end{proof}

\paragraph{Unprojection of $\D$}
Unprojecting the divisor $\D$ in $Y$ gives a new 3-fold $X \subset \P(1^3,3^4,5)$:
\begin{displaymath}
\xymatrix@C=-3pt{
                 &           &          &  X                         & \subset & \P(1,1,1,3,3,3,3,5)\\
                 &  \D        &\subset &  Y         \ar@{-->}[u]    & \subset & \hspace{-0.4cm}\P(1,1,1,3,3,3,5)\\
    \E, &  \D        &\subset &  Z_{6,6}   \ar@{-->}[u]    & \subset & \hspace{-0.8cm}\P(1,1,1,3,3,3)\\
}
\end{displaymath}
The unprojection variable, denoted $s$, has weight~3.
Note that $ P_s = \frac{1}{3}(1,1,1) \in X'$.
The proof of Theorem~\ref{thm!unprojto4} now applies to show the following.

\begin{prop}
$X$ is a quasismooth Calabi-Yau 3-fold with two quasismooth singularities each of type $ \frac{1}{3}(1,1,1)$.
\end{prop}

\subsection{Construction of $\text{Jer}_{45}$}

We describe the construction in the case, stating the intermeditate steps;
the proofs follow those of~\S\ref{s!tom1}.
Fix coordinates $x,y,z,t,u,v$ on $\P(1^3,3^3)$ and
define $\D=\P^2$ by $I_\D = (t, u, v)$.
Define $\E$ by $I_\E = (x, u + p_3, t + q_3)$ where
$p_3$ and $q_3$ are general cubics in $y,z$.
Since $p_3$ and $q_3$ are general, $\E$ and $\D$ are disjoint.

Let $Z'=Z'_{6,6}$ be a general $(6,6)$ complete intersection that contains $\D\cup\E$.
We may write the equations of $Z'$ as
\begin{eqnarray*}
\begin{pmatrix}
A_5 & B_3 & C_3 \\
D_5 & E_3 & F_3
\end{pmatrix}
\begin{pmatrix}
x \\
u + p_3 \\
t + q_3
\end{pmatrix} =
\begin{pmatrix}
0 \\
0
\end{pmatrix},
\end{eqnarray*}
where $A,\dots,F$ are general polynomials of the indicated degrees that lie in $I_\D$.

\begin{lem}
In the notation above, the singular locus $\Sing(Z')$ of $Z'$ consists of exactly 27 points on $\D$ and 13
points on $\E$, and all singular points are ordinary nodes.
\end{lem}

Unprojecting $\E\subset Z'$ gives a 3-fold $Y' \subset \P(1^3,3^3,5)$ which
contains the divisor $\D = (t=u=v=w=0)$.
The unprojection variable has weight 5 and is denoted by $w$.
The equations of $Y'$ can be realised as the five maximal Pfaffians of the antisymmetric
$5\times 5$ matrix
\begin{displaymath}
\begin{pmatrix}
x & u + p_3(y,z) & C_3 & F_3 \\
      & t + q_3(y,z) & -B_3 & -E_3 \\
  &               & A_5 & D_5 \\
  &        &        & w
\end{pmatrix}
\end{displaymath}
which is in $\text{Jer}_{45}$ format for $\D$.

\begin{prop}
$Y'$ is quasismooth outside $\D$, and has in total 27 isolated singularities on $\D$,
all of which are nodes.
\end{prop}

Unprojecting $\D\subset Y'$ gives a new 3-fold $X' \subset \P(1^3,3^4,5)$.
The unprojection variable, denoted $s$, has weight~3, and
$ P_s = \frac{1}{3}(1,1,1)\in X'$.

\begin{prop}
$X'$ is a quasismooth Calabi-Yau 3-fold with two quasismooth singularities
each of type $ \frac{1}{3}(1,1,1)$.
\end{prop}

Theorem~\ref{main theorem}(\ref{thm!part}) now follows: the two Calabi--Yau 3-folds
$X$ and $X'$ have the same Hilbert series (they are calculated by identical
orbifold Riemann--Roch formulas), but have different
Euler characteristics since the number of nodes on $\D\subset Y$ and $\D\subset Y'$
differ; therefore they are members of two different deformation families.

The remaining cases of Theorem~\ref{main theorem} in codimensions~3 and~4
are proved by the same method. The key point is the starting configurations
$\D\cup\E\subset Z \subset w\P^5$.

In case (a), we start with $\D=(u=v=w=0)$ and
$\E=(u-\alpha = v-\beta = w-\gamma=0)$ inside an otherwise
general $Z_{6,6}\subset\P(1^3,3^3)$ with coordinates $x, y, z, u, v, w$, where
$\alpha$, $\beta$ and $\gamma$ are cubics in $x,y,z$.
If $\alpha$, $\beta$ and $\gamma$ are all chosen generally, then
$\D$ and $\E$ are disjoint; if there is a single linear relation among them,
then $\D\cap\E$ is 9 points, the intersection of two cubics in~$\P^2$.
After unprojecting $\E\subset Z$ and mounting the equations of the
resulting $Y\subset\P(1^3,3^4)$
in an antisymmetric matrix, as in~\eqref{e!j12}, these two cases
correspond to Jer$_{45}$ and Tom$_3$, respectively, and the
number of nodes on $Y$ differs by 9 between the two cases,
being those nodes lying at $\D\cap\E$,
which are resolved by the unprojection of~$\E$.

In case (c), in the same ambient $\P(1^3,3^3)$, with $\D=(z=v=w=0)$,
the three cases arise from the choice of
$\E=(z=u=v-\alpha_3(x,y)=0)$ or $(y=u=v-\alpha=0)$ with either $\alpha=0$ or~$x^3$.
Case (d) is similar to (a) but working with $Z_{6,8}\subset\P(1^3,3^2,5)$.

\section{Final remarks}
\label{s!comments}

Of course this is just the start.
The web of Calabi--Yaus is vast: compare with Batyrev--Kreuzer \cite{BK},
who navigate the web using conifold transitions starting from
toric hypersurfaces and find over 30,000 cases.
With our approach, the initial parameters are $P_1$, $P_2$ and the basket
of singularities $\cB$ (which we have taken to be isolated quotient singularities).
For example, varying $P_1$ and $P_2$ with these same possible baskets
$\cB=\left\{ n \times \frac{1}3(1,1,1), m \times \frac{1}5(1,1,3) \right\}$ gives
other candidates. In cases where $P_1\ge2$ and $m,n\ge1$ these are as follows:
\[
\begin{array}{ccccc|rcl}
P_1 & P_2 & n & m & \text{codim} &
  \multicolumn{3}{c}{\text{Calabi--Yau 3-fold}}  \\
\hline
2 & 5 & 1 & 3 & 4 & X_{8^3,9^3,10^3} &\subset& \P(1,1,3,3,4,5,5,5) \\
3 & 6 & & & & \multicolumn{3}{c}{\text{See Theorem~\ref{main theorem}}} \\
3 & 7 & 1 & 1 & 3 & X_{5,6^2,7,8} &\subset& \P(1,1,1,2,3,3,5)  \\
3 & 7 & 2 & 1 & 4 & X_{5^2, 6^4, 7, 8^2} &\subset& \P(1,1,1,2,3,3,3,5)  \\
4 & 10 & 1 & 1 & 3 & X_{4,6^3,8} &\subset& \P(1,1,1,1,3,3,5) \\
4 & 10 & 1 & 1 & 4 & X_{4^2, 6^5, 8^2} &\subset& \P(1,1,1,1,3,3,3,5) \\
4 & 11 & 1 & 1 & 4 &  X_{4^2, 5^2, 6^3, 7, 8} &\subset& \P(1,1,1,1,2,3,3,5) \\
5 & 15 & 1 & 1 & 4 & X_{4^4, 6^4, 8} &\subset& \P(1,1,1,1,1,3,3,5)
\end{array}
\]
In codimension~3 we construct a single deformation family, while again in codimension~4
we construct at least two deformation families in each case.

It is worth noting, finally, that 
we go to some trouble in \S\ref{s!codim3} to prove existence and quasismoothness even
in codimension~3.
Our caution is justified by an example of a Calabi--Yau 3-fold
$Y\subset\P(1,1,2,5,8,13,19)$ given in \cite[\S4.2.2]{BKZ}.
Any general such $Y$ fails to be quasismooth at a single point, where it
has a node: $Y$ is not (quasi)smoothable within its Pfaffian format.
This is striking, since it is easy to check that $Y$ is quasismooth
along all nontrivial orbifold strata of the ambient space, and the
common experience is that smoothness elsewhere is then routine.

Applying the methods of \S\ref{s!codim4},
the codimension~3 Calabi--Yau $Y$ admits an unprojection to a
quasismooth Calabi--Yau $X\subset\P(1,1,2,5,8,13,19,37)$ in codimension~4
which has a single quotient singularity, $\frac{1}{37}(5,13,19)$.
But, unlike all cases here, this arises only as a single Jerry unprojection,
and so provides only one deformation family of such~$X$.

\bigskip

\noindent Gavin Brown, \\
Mathematics Institute, University of Warwick, 
Coventry CV4 7AL, UK

\noindent Email: G.Brown@warwick.ac.uk

\bigskip

\noindent Konstantinos Georgiadis, \\
Department of Mathematical Sciences, Loughborough University, 
LE11 3TU, UK

\noindent Email: K.Georgiadis@lboro.ac.uk


\begin{thebibliography}{99}

\bibitem{ABR}
S. Alt{\i}nok, G. Brown, M. Reid
{\em Fano 3-folds, K3 surfaces and graded rings},
Contemp. Math. {\bf 314} (2002) 25--61

\bibitem{B} V. V. Batyrev, \emph{Birational Calabi--Yau varieties have equal Betti numbers}, in New trends in algebraic geometry, Warwick, 1996,
London Math. Soc. Lecture Note Ser., \textbf{264}, 1999

\bibitem{BK}
{V.V. Batyrev, M. Kreuzer},
{Constructing new {C}alabi-{Y}au 3-folds and their mirrors via conifold transitions},
{Adv. Theor. Math. Phys.} {\bf 14}:3 (2010) {879--898}

\bibitem{Ber} M. Bertin, \emph{Examples of Calabi--Yau 3-folds of $\P^7$ with $\rho = 1$}, Canad. J. Math. \textbf{61}, no. 5, 1050-1072, 2009

\bibitem{CGGK}
{T-M. Chiang, B.R. Greene, M. Gross, Y. Kanter},
{Black hole condensation and the web of {C}alabi-{Y}au manifolds},
{$S$-duality and mirror symmetry (Trieste, 1995)},
{Nuclear Phys. B Proc. Suppl.} {\bf 46} (1996) {82--95}

\bibitem{BCG}
G. Brown, S. Coughlan, K. Georgiadis,
{\em Calabi--Yau threefolds with $\Z/3$ orbifold points},
in preparation.

\bibitem{grdb}
G. Brown, A.M. Kasprzyk et al,
Graded Ring Database,
online at {\tt grdb.co.uk}

\bibitem{Bk3}
G. Brown,
{\em A database of polarised K3 surfaces}, Exp. Math. {\bf 16}:1 (2007) 7--20

\bibitem{BKZ}
G. Brown, A.M. Kasprzyk, L. Zhu
{\em Gorenstein formats, canonical and CalabiÐYau 3-folds},
25pp. Available at arXiv:1409.4644

\bibitem{TJ}
G. Brown, M. Kerber, M. Reid,
\emph{Fano 3-folds in codimension 4, Tom and Jerry, Part I}, 
Compos. Math. \textbf{148} (2012), 1171--1194

\bibitem{BRZ} A. Buckley, M. Reid, S. Zhou, \emph{Ice cream and orbifold Riemann--Roch}, Izv. Ross. Akad. Nauk Ser. Mat. \textbf{77}, no. 3, 29-54, 2013;
translation in Izv. Math. \textbf{77}, no. 3, 461--486, 2013

\bibitem{Clemens} H. Clemens, \emph{Double solids}, Adv. in Math. \textbf{47}, 107--230, 1983

\bibitem{CPR}
{A. Corti, A. Pukhlikov, M. Reid},
{Fano {$3$}-fold hypersurfaces},
in {Explicit birational geometry of 3-folds},
  {LMS Lecture Note Ser.}    {\bf 281} CUP (2000)
 {175--258}

\bibitem{CR}
A. Corti, M. Reid,
{\em Weighted Grassmannians}, in Algebraic geometry, de Gruyter, Berlin,
(2002) 141--163

\bibitem{Ful} W. Fulton, \emph{Algebraic curves. An introduction to algebraic geometry}, Addison-Wesley Publishing Company, 1989

\bibitem{Geo} K. Georgiadis, \emph{Polarized Calabi--Yau threefolds in codimension 4},
PhD thesis, Loughborough University, December 2014, 87pp.

\bibitem{GW} S. Goto, K. Watanabe, \emph{On graded rings I}, J. Math. Soc. Japan \textbf{30}, no. 2, 179--213, 1978

\bibitem{GH}
{P.S. Green, T. H{\"u}bsch},
{Connecting moduli spaces of {C}alabi-{Y}au threefolds},
{Comm. Math. Phys.} {\bf 119}:3 (1988) {431--441},
     
\bibitem{KS}
M. Kreuzer, H. Skarke,
{\em Complete classification of reflexive polyhedra in four dimensions},
Adv. Theor. Math. Phys., 4(6) (2000) 1209--1230

\bibitem{gross}
M. Gross,
{The deformation space of {C}alabi-{Y}au {$n$}-folds with canonical singularities can be obstructed} in {\em Mirror symmetry, {II}}, {AMS/IP Stud. Adv. Math.} {\bf 1},
{Amer. Math. Soc., Providence, RI},
{401--411}, {1997},

\bibitem{Prim} M. Gross, \emph{Primitive Calabi--Yau 3-folds}, J. Diff. Geom. \textbf{45}, 288--318, 1997


\bibitem{P} S.A. Papadakis,
{\em Kustin-{M}iller unprojection with complexes},
 {J. Alg. Geom.} {\bf 13}:2 (2004) {249--268}

\bibitem{PR} S. Papadakis and M. Reid, \emph{Kustin-Miller unprojection without complexes}, J. Alg. Geom. \textbf{13} (2004) 563--577

\bibitem{R2}
M. Reid,
{The moduli space of {$3$}-folds with {$K=0$} may nevertheless be irreducible},
{Math. Ann.} {\bf 278} (1987), {329--334}

\bibitem{Rgraded}
M. Reid,
Graded rings and birational geometry, in Proc. of algebraic
geometry symposium (Kinosaki, Oct 2000), K. Ohno (Ed.), 1--72

\bibitem{tonoli}
F. Tonoli.
{\em Construction of Calabi-Yau 3-folds in $\mathbb{P}^6$},
J. Alg. Geom., {\bf 13}:2 (2004) 209--232

\end{thebibliography}
\end{document}